\documentclass[11pt,reqno]{amsart}

\usepackage[T1]{fontenc}
\usepackage{amsmath,amssymb,amsthm}
\usepackage{microtype}
\usepackage{tikz}
\usetikzlibrary{arrows.meta}
\usepackage[section]{placeins}
\usepackage[margin=1.25in]{geometry}
\usepackage[colorlinks=true,linkcolor=blue,citecolor=blue,urlcolor=blue]{hyperref}
\hypersetup{
  pdftitle={Finitely generated positive cones in Fn x Z},
  pdfauthor={Hang Lu Su},
  pdfkeywords={left-orderable group, positive cone, isolated left-order, free group,
    finite-index subgroup}
}

\theoremstyle{plain}
\newtheorem{theorem}{Theorem}[section]
\newtheorem{lemma}[theorem]{Lemma}
\newtheorem{proposition}[theorem]{Proposition}
\newtheorem{corollary}[theorem]{Corollary}

\theoremstyle{definition}
\newtheorem{remark}[theorem]{Remark}

\newcommand{\Z}{\mathbb{Z}}
\newcommand{\LO}{\operatorname{LO}}
\newcommand{\inv}{^{-1}}

\DeclareMathOperator{\Stab}{Stab}
\DeclareMathOperator{\Sym}{Sym}

\begin{document}

\title[Finitely generated positive cones in $F_n\times\Z$]
      {Finitely generated positive cones in $F_n \times \Z$}

\author{Hang Lu Su}

\subjclass[2020]{Primary 20F60; Secondary 06F15, 20E05}
\keywords{left-orderable group, positive cone, isolated left-order, free group,
finite-index subgroup}

\begin{abstract}
We construct, for every even $n\ge2$, a positive cone on $F_n\times\Z$ that is finitely generated as a semigroup, extending the previously known construction for $n=2$~\cite{Su2020}.
Malicet, Mann, Rivas and Triestino proved that $F_n\times\Z$ admits an isolated left-order if and only if $n$ is even~\cite{MMRT}.
Since every finitely generated positive cone determines an isolated left-order, for $n\ge2$, the group $F_n\times\Z$ admits a finitely generated positive cone if and only if $n$ is even.
\end{abstract}

\maketitle

\section{Introduction}

A group $G$ is \emph{left-orderable} if it admits a total order invariant under left multiplication.
Such an order is encoded by its \emph{positive cone}: a subsemigroup $P\subseteq G$ for which
\[
  G=P\sqcup P\inv\sqcup\{1\}.
\]
For $S\subseteq G$, write $\langle S\rangle^+$ for the semigroup of non-empty
products $s_1\cdots s_\ell$ with $s_i\in S$ and $\ell\ge1$.
If $P=\langle S\rangle^+$ for a finite set $S$, we call $P$ \emph{finitely generated}.

Non-abelian free groups admit no finitely generated positive cones by~\cite[Corollary~2]{HermillerSunic}.
Taking a direct product with $\Z$ changes the picture.
Mann and Rivas constructed an isolated left-order on $F_2\times\Z$~\cite[Corollary~5.5]{MannRivas}, and $F_2\times\Z$ in fact admits a finitely generated positive cone~\cite[Corollary~1.3]{Su2020}.
Malicet, Mann, Rivas and Triestino subsequently proved that $F_n\times\Z$ has an isolated left-order if and only if $n$ is even~\cite[Theorem~1.1]{MMRT}.

A finitely generated positive cone determines an isolated left-order (Lemma~\ref{lem:isolated}).
Thus the result of Malicet, Mann, Rivas and Triestino obstructs finite generation when $n$ is odd.
Our main result is the corresponding existence statement in every even rank.

\begin{theorem}\label{thm:main}
For every even integer $n\ge2$, the group $F_n\times\Z$ has a positive cone which is finitely generated as a semigroup.
\end{theorem}

Combining Theorem~\ref{thm:main} with the known obstruction above gives the full classification.

\begin{corollary}\label{cor:parity}
For $n\ge2$, the group $F_n\times\Z$ has a finitely generated positive cone if and only if $n$ is even.
\end{corollary}

\begin{proof}
The even case is Theorem~\ref{thm:main}.
For odd $n$, Lemma~\ref{lem:isolated} would turn a finitely generated positive cone into an isolated left-order, contradicting~\cite[Theorem~1.1]{MMRT}.
\end{proof}

We prove Theorem~\ref{thm:main} through the subgroup tower
\[
  K\cong F_n\times\Z \ \le \ H\cong F_2\times\Z \ \le \ \Gamma_2,
  \qquad \Gamma_2=\langle a,b\mid ba^2b=a\rangle.
\]
We give a criterion, based on Reidemeister--Schreier rewriting (see~\cite[Section~2.3]{MagnusKarrassSolitar}), for passing a finitely generated positive cone to a finite-index subgroup. We apply it first to the Dubrovina--Dubrovin cone $\langle a,b\rangle^+$ on $\Gamma_2$~\cite[Theorem~5]{DubrovinaDubrovin} to obtain a finitely generated positive cone on $H$, and then apply it again to obtain one on $K\cong F_n\times\Z$.

The construction of the cone on $H$ appears in~\cite[Proposition~5.2 and proof of Corollary~1.3]{Su2020}. The new contribution is the cyclic-cover construction in Section~\ref{sec:cover}, which enables the second application uniformly for every even rank $n\ge4$. This construction first appeared, with a longer case-by-case coset calculation, in the author's thesis~\cite[Chapter~12]{SuThesis}; the cycle computation here makes the role of parity explicit.

\section{Passing finitely generated positive cones to finite-index subgroups}\label{sec:lifting}

Positive cones correspond bijectively to left-orders: given $P$, define $g \prec h$ if and only if $g\inv h\in P$.
We identify the space $\LO(G)$ of left-orders with the corresponding subspace of $\{0,1\}^G$.
Thus the sets $\{P:g\in P\}$, for $g\in G$, form a subbasis for its topology.

We first record the elementary observation used to obtain the odd-rank obstruction in Corollary~\ref{cor:parity}.

\begin{lemma}\label{lem:isolated}
If a positive cone $P$ of $G$ is finitely generated as a semigroup, then $P$ is an isolated point of $\LO(G)$.
\end{lemma}

\begin{proof}
Write $P=\langle S\rangle^+$ with $S$ finite.
The set $V=\{Q\in\LO(G):S\subseteq Q\}$ is open because $S$ is finite.
For $Q\in V$, closure under multiplication gives $P\subseteq Q$.
If this inclusion were strict, some $q\in Q\setminus P$ would satisfy $q\inv\in P\subseteq Q$, contradicting the trichotomy for $Q$.
Hence $V=\{P\}$.
\end{proof}

It remains to construct a cone for even $n$.

If $P$ is a positive cone on $G$ and $K\le G$, then $K\cap P$ is automatically a positive cone on $K$.
Thus the difficulty is not defining the cone on $K$, but proving that it remains finitely generated.
The following criterion gives a sufficient condition: under a suitable positivity hypothesis, the Reidemeister--Schreier generators for $K$ also generate $K\cap P$ as a semigroup.
It abstracts the rewriting argument used in~\cite[proof of Proposition~5.2]{Su2020}.
Let $K\le G$ and let $T$ be a right transversal containing $1$, so that $G=\bigsqcup_{t\in T}Kt$.
For $g\in G$, denote by $\overline g$ the unique element of $T$ such that $Kg=K\overline g$.
For $t\in T$ and $s\in G$, set
\[
  \gamma(t,s):=ts\,\overline{ts}\inv\in K.
\]
Indeed, $Kts=K\overline{ts}$, so $ts\,\overline{ts}\inv\in K$.

\begin{lemma}\label{lem:lifting}
Let $P$ be a positive cone of $G$ with $P = \langle S \rangle^{+}$ for a finite set $S \subseteq G$.
Let $K \le G$, and let $T$ be a right transversal for $K$ in $G$ with $1 \in T$.
Suppose that
\[
  \gamma(t,s)\in P\cup\{1\}
  \qquad(t\in T,\ s\in S).
\]
Then, setting $S_K:=\{\gamma(t,s):t\in T,\ s\in S\}\setminus\{1\}$,
\[
  K\cap P=\langle S_K\rangle^+.
\]
In particular, if $K$ has finite index in $G$, then $K\cap P$ is a finitely generated positive cone of $K$.
\end{lemma}

\begin{proof}
The intersection $K\cap P$ is a positive cone of $K$: closure is immediate, and for every $k\in K\setminus\{1\}$, exactly one of $k$ and $k\inv$ lies in $P$.
Every $\gamma(t,s)$ lies in $K$, and by definition
of $S_K$, $S_K\subseteq P$. Hence $S_K\subseteq K\cap P$, and therefore
$\langle S_K\rangle^+\subseteq K\cap P$.

To show that $K\cap P\subseteq\langle S_K\rangle^+$, write $S^+$ and $S_K^+$ for the
sets of non-empty words over $S$ and over $S_K$, and let
$\pi:(S\cup S_K\cup\{1\})^*\to G$ be the standard evaluation map. We will show
that for every word $w\in S^+$ whose image $\pi(w)$ is in $K$, there is
a word $v\in S_K^+$ such that $\pi(v)=\pi(w)$.

Write $w=s_1\cdots s_\ell$ and let $w_i=s_1\cdots s_i$, with $w_0$ the
empty word. The Reidemeister--Schreier rewriting map $\tau$ gives
\[
  \tau(w)=\prod_{i=1}^{\ell}
  \gamma\bigl(\overline{\pi(w_{i-1})},s_i\bigr),
  \qquad \pi(\tau(w))=\pi(w),
\]
because $\pi(w)\in K$. Every non-identity factor in $\tau(w)$ belongs to
$S_K$. After deleting the identity factors, we therefore obtain a word
$v$ over $S_K$ with $\pi(v)=\pi(w)$. This word is non-empty because
$\pi(w)\in P$, so $\pi(w)\ne1$. Hence $v\in S_K^+$, proving the desired
inclusion.
Finally, if $K$ has finite index in $G$, then $T$ is finite. Since $S$ is
finite, so is $S_K$.
\end{proof}

\section{A positive cone on \texorpdfstring{$F_2\times\Z$}{F2 x Z}}\label{sec:navas}

This section recalls the $F_2\times\Z$ construction from~\cite[Section~5]{Su2020} and fixes notation for the higher-rank extension.
In particular, the subgroup $H$, the generating set $Y$, and the isomorphism $\psi$ are all already present there.
We include the details to make the later argument self-contained.

Set
\[
  \Gamma_2:=\langle a,b\mid ba^2b=a\rangle\cong B_3,
  \qquad
  a=\sigma_1\sigma_2,\quad b=\sigma_2\inv.
\]

\begin{theorem}[{Dubrovina--Dubrovin~\cite[Theorem~5]{DubrovinaDubrovin}; see also Navas~\cite[Main Theorem]{Navas}}]\label{thm:navas}
The subsemigroup $P_2:=\langle a,b\rangle^+$ is a positive cone of $\Gamma_2$.
\end{theorem}

Observe that
\[
  b\inv a=a^2b,
  \qquad
  ab\inv=ba^2.
\]

The one-sided case $b^{-s}a\in P_2$ for $s\ge0$ is~\cite[Lemma~5.1]{Su2020}; we record the two-sided extension needed below.

\begin{lemma}\label{lem:bab}
For all $p,q\in\Z$, the element $b^p a b^q$ belongs to $P_2$.
Moreover, $b^k\in P_2$ for every $k\ge1$.
\end{lemma}

\begin{proof}
Induction using these identities gives, for $s,t\ge0$,
\[
  b^{-s}a=a(ab)^s,
  \qquad
  ab^{-t}=(ba)^t a.
\]
These identities handle the cases in which at most one of $p,q$ is negative.
If $p=-s<0$ and $q=-t<0$, then
\[
  b^{-s}ab^{-t}=a(ab)^{s-1}(ba)^{t-1}a,
\]
which is again a positive word.
The assertion about $b^k$ is immediate.
\end{proof}

As in~\cite[proof of Corollary~1.3]{Su2020}, let
\[
  \varphi:\Gamma_2\longrightarrow\Z/6\Z,
  \qquad
  \varphi(a)=4,\quad \varphi(b)=1.
\]
This defines an epimorphism: the relation is respected because
$\varphi(ba^2b)=1+2\cdot4+1\equiv4=\varphi(a)\pmod 6$,
and $\varphi$ is onto since $\varphi(b)=1$.
Thus $H:=\ker\varphi$ has index six.
The next proposition collects the specialization $(n,m,\mu)=(2,6,4)$ of~\cite[Proposition~5.2 and proof of Corollary~1.3]{Su2020} in the notation needed to pass this cone to a second subgroup.
\begin{proposition}[A positive cone on $H$]\label{prop:base-cone}
Write
\[
  F_2\times\Z=\langle x,y,z\mid[x,z], [y,z]\rangle.
\]
There is an isomorphism $\psi:H\to F_2\times\Z$ such that
\[
  \psi(ab^2)=x,
  \qquad
  \psi(a^2b^2a^2)=y,
  \qquad
  \psi(a^3)=z.
\]
Moreover,
\[
  Y:=\{ab^2,\ b\inv ab^3,\ b^{-2}ab^4,\ b^{-3}ab^5,
        \ b^{-4}a,\ b^{-5}ab,\ b^6\}
\]
generates the positive cone
\[
  P_H:=H\cap P_2=\langle Y\rangle^+
\]
of $H$.
\end{proposition}

We also write $x=ab^2$, $y=a^2b^2a^2$ and $z=a^3$ for these elements of
$\Gamma_2$ themselves, so that $\psi$ carries each to the generator of
$F_2\times\Z$ bearing the same name in the sequel. 

\begin{proof}
\emph{Identifying the subgroup.}
The identification in~\cite[proof of Corollary~1.3]{Su2020} was
originally obtained using GAP. We sketch the computer-free argument
from~\cite[Section~12.4.2]{SuThesis}. Following the notation of the
thesis, write $f=a$ and $h=b\inv a$. In these generators,
\[
  \Gamma_2=\langle f,h\mid f^3=h^2\rangle\cong B_3.
\]
The common power $f^3=h^2=a^3=z$ generates the centre and has infinite
order. Hence
\[
  Q:=\Gamma_2/\langle z\rangle
    \cong C_3*C_2\cong\operatorname{PSL}_2(\Z).
\]
For $g\in\Gamma_2$, write $\bar g$ for its image in $Q$.
Since $\varphi(z)=0$, the map $\varphi$ induces
$\bar\varphi:Q\to\Z/6\Z$, with
$\bar\varphi(\bar f)=4$ and $\bar\varphi(\bar h)=3$.
Under the identification $Q_{\mathrm{ab}}\cong C_3\times C_2\cong C_6$,
this is the abelianisation map, so the image of $H$ in $Q$ is
$\ker\bar\varphi=[Q,Q]$. A direct calculation shows that
$\bar x$ and $\bar y$ generate this subgroup, while ping--pong for the standard action on
$\mathbb P^1(\mathbb R)$ shows that they generate it freely.
Thus $\langle x,y\rangle\to Q$ is injective and
$H/\langle z\rangle=\langle\bar x,\bar y\rangle$. Since $z$ is central,
\[
  H=\langle x,y\rangle\times\langle z\rangle\cong F_2\times\Z.
\]
This is the asserted isomorphism $\psi$.

\emph{Computing $P_H=H\cap P_2$.}
This is~\cite[Proposition~5.2]{Su2020}, specialised to $\Gamma_2$ and $\varphi$; we recover it from Lemma~\ref{lem:lifting}.
Take the right transversal
\[
  T:=\{1,b\inv,b^{-2},b^{-3},b^{-4},b^{-5}\}
\]
for $H$ in $\Gamma_2$.
It is a transversal because $\varphi(b^{-s})=-s$ runs through every residue modulo six as $0\le s\le5$.
For $t=b^{-s}$, where $0\le s\le5$, the representative of $ta$ is $b^{-r}$ with
$r\equiv s-4\pmod 6$ and $0\le r\le5$.
Thus the Schreier elements obtained from $a$ are the first six elements of $Y$,
and they lie in $P_2$ by Lemma~\ref{lem:bab}.

For the letter $b$, the Schreier element is trivial when $1\le s\le5$, while
$\gamma(1,b)=b^6\in P_2$.
The non-trivial Schreier elements are therefore exactly those in $Y$, and Lemma~\ref{lem:lifting} gives $P_H=\langle Y\rangle^+$.
\end{proof}

\section{The cyclic cover}\label{sec:cover}

We now begin the higher-rank extension beyond~\cite[Section~5]{Su2020}.
To pass this positive cone to a further subgroup, we will use powers of $b^6$ as coset representatives.
The defining relation and the centrality of $z=a^3$ give
\[
  b^6=(a^2b^2a^2)\inv(ab^2)\inv(a^2b^2a^2)(ab^2)a^{-3},
\]
so that $\psi(b^6)=y\inv x\inv yx z\inv$.
Its image in the free factor of $F_2\times\Z$ is therefore
\[
  g:=y\inv x\inv yx\in F_2=\langle x,y\rangle.
\]
We now construct an index-$(n-1)$ subgroup $K_n\le F_2$ for which the action of $g$ on right cosets is transitive exactly when $n$ is even.

Fix $n\ge2$ and let
$V_n:=\{v_0,v_1,\dots,v_{n-2}\}$.
Define the following elements of $\Sym(V_n)$:
\[
  \sigma_x
    := \prod_{j=1}^{\lfloor (n-2)/2\rfloor}
       (v_{2j-1}\,v_{2j}),
  \qquad
  \sigma_y
    := \prod_{j=1}^{\lfloor (n-1)/2\rfloor}
       (v_{2j-2}\,v_{2j-1}),
\]
where an empty product denotes the identity.
Thus $\sigma_x^2=\sigma_y^2=1$.
Moreover, $\sigma_x$ fixes $v_0$, and, for $n\ge3$, the terminal vertex $v_{n-2}$ is fixed by $\sigma_y$ when $n$ is even and by $\sigma_x$ when $n$ is odd.

These permutations define a right action of $F_2=\langle x,y\rangle$ on $V_n$ by
\[
  v\cdot x:=\sigma_x(v),\qquad v\cdot y:=\sigma_y(v).
\]
Since $\sigma_x^2=\sigma_y^2=1$, the inverse letters act in the same way.
Set
\[
  K_n:=\Stab_{F_2}(v_0).
\]
The corresponding labelled cover is shown schematically in Figure~\ref{fig:cover}.

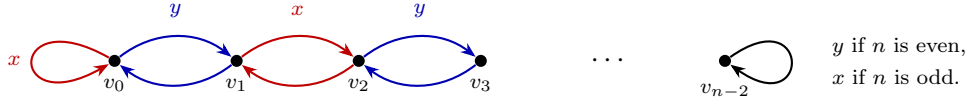
\begin{figure}[!htb]
\centering
\begin{tikzpicture}[
    >={Stealth[length=2mm]},
    vertex/.style={circle,fill=black,inner sep=1.6pt},
    xedge/.style={red!75!black,thick},
    yedge/.style={blue!70!black,thick},
    scale=0.95
  ]
  \foreach \i/\lab in {0/{v_0},1/{v_1},2/{v_2},3/{v_3}} {
    \node[vertex] (v\i) at (1.7*\i,0) {};
    \node[below=3pt] at (1.7*\i,0) {\scriptsize $\lab$};
  }
  \node[vertex] (vt) at (8.5,0) {};
  \node[below=5pt] at (vt) {\scriptsize $v_{n-2}$};
  \node at (6.9,0) {$\cdots$};
  \draw[xedge,->] (v0) .. controls (-1.5,0.85) and (-1.5,-0.85) .. (v0);
  \node[red!75!black] at (-1.40,0) {\scriptsize $x$};
  \draw[yedge,->] (v0) to[bend left=38] (v1);
  \draw[yedge,->] (v1) to[bend left=38] (v0);
  \node[blue!70!black] at (0.85,0.7) {\scriptsize $y$};
  \draw[xedge,->] (v1) to[bend left=38] (v2);
  \draw[xedge,->] (v2) to[bend left=38] (v1);
  \node[red!75!black] at (2.55,0.7) {\scriptsize $x$};
  \draw[yedge,->] (v2) to[bend left=38] (v3);
  \draw[yedge,->] (v3) to[bend left=38] (v2);
  \node[blue!70!black] at (4.25,0.7) {\scriptsize $y$};
  \draw[thick,->] (vt) .. controls (9.7,0.85) and (9.7,-0.85) .. (vt);
  \node[align=left,anchor=west] at (9.85,0)
    {\scriptsize $y$ if $n$ is even,\\[-1pt]\scriptsize $x$ if $n$ is odd.};
\end{tikzpicture}
\caption{Schematic form of the $(n-1)$-sheeted cover for $n\ge6$.
The cases $3\le n\le5$ are obtained by truncating the chain, while for $n=2$ there is a single vertex.
The edge labels alternate, with an $x$-loop at $v_0$; the label of the terminal loop records the parity of $n$.}
\label{fig:cover}
\end{figure}
\FloatBarrier

Equivalently, these permutations encode a labelled $(n-1)$-sheeted cover of the two-petal rose.
Following a $y$-edge and then an $x$-edge gives the permutation $\theta$ used below.

\begin{proposition}\label{prop:cycle}
Let $\theta\in\Sym(V_n)$ be defined by
\[
  \theta(v):=v\cdot(yx)=\sigma_x(\sigma_y(v)).
\]
Then:
\begin{enumerate}
  \item $v\cdot g=\theta^2(v)$ for every $v\in V_n$;
  \item $\theta$ is an $(n-1)$-cycle;
  \item the action of $F_2$ on $V_n$ is transitive, and $K_n$ has index
  $n-1$ in $F_2$ and is free of rank $n$; thus $K_n\cong F_n$;
  \item $\langle g \rangle$ acts transitively on $V_n$ if and only if $n$ is even.
\end{enumerate}
\end{proposition}

\begin{proof}
(1) Reading the word $g=y\inv x\inv yx$ from left to right applies $\sigma_y,\sigma_x,\sigma_y,\sigma_x$, because $\sigma_x^2=\sigma_y^2=1$.
This is $\theta^2$.

(2) When $n=2$, we have $V_2=\{v_0\}$, so $\theta$ is the identity
permutation on $V_2$. For $n=3,4,5$, direct calculation gives, respectively,
$(v_0\ v_1)$, $(v_0\ v_2\ v_1)$, and $(v_0\ v_2\ v_3\ v_1)$.
Now suppose that $n\ge6$. Since $\theta$ applies
$\sigma_y$ and then $\sigma_x$, the staggered swaps move two places to the
right on even-indexed vertices and two places to the left on odd-indexed
vertices, except at the turns: $\theta(v_1)=v_0$, while
$\theta(v_{n-2})=v_{n-3}$ if $n$ is even and
$\theta(v_{n-3})=v_{n-2}$ if $n$ is odd. Hence
\[
\begin{array}{ll}
\theta=(v_0\ v_2\ \cdots\ v_{n-2}\ v_{n-3}\ v_{n-5}\ \cdots\ v_1),
  & n\text{ even},\\[2mm]
\theta=(v_0\ v_2\ \cdots\ v_{n-3}\ v_{n-2}\ v_{n-4}\ \cdots\ v_1),
  & n\text{ odd}.
\end{array}
\]
In either case every vertex of $V_n$ appears exactly once, so $\theta$ is an
$(n-1)$-cycle.

(3) Since $\theta$ is the action of $yx\in F_2$ and is an $(n-1)$-cycle, the
$F_2$-action is transitive. Orbit--stabiliser gives
$[F_2:K_n]=|V_n|=n-1$, and the Nielsen--Schreier formula gives
$\operatorname{rank}(K_n)=(n-1)(2-1)+1=n$.

(4) By (1) and (2), $g$ acts as the square of an $(n-1)$-cycle.
The square of an $(n-1)$-cycle is transitive exactly when $n-1$ is odd, or equivalently when $n$ is even.
\end{proof}

\begin{remark}
Concretely, $v_i$ corresponds to the right coset $K_n u_i$, where $u_i$ is the alternating word $yxyx\cdots$ of length $i$.
\end{remark}

\begin{remark}\label{rmk:odd-transversal}
For odd $n$, Proposition~\ref{prop:cycle} shows only that the powers of $b^6$ do not form a transversal: the projected action has two orbits.
The nonexistence of a finitely generated positive cone for odd $n$ comes from~\cite[Theorem~1.1]{MMRT}, not from this computation.
\end{remark}

\section{The construction for even \texorpdfstring{$n$}{n}}\label{sec:main}

Proposition~\ref{prop:base-cone} supplies the positive generating set $Y$, which contains $b^6$; Proposition~\ref{prop:cycle} supplies the power transversal when $n$ is even.
We now apply Lemma~\ref{lem:lifting} a second time.

\begin{proof}[Proof of Theorem~\ref{thm:main}]
Fix an even integer $n\ge2$.
Let $\rho:F_2\times\Z\to F_2=\langle x,y\rangle$ be the projection and define
\[
  K:=\psi\inv\bigl(\rho\inv(K_n)\bigr)\le H.
\]
By Proposition~\ref{prop:cycle}(3), $K_n\cong F_n$ and
$[F_2:K_n]=n-1$.
Since $\rho\inv(K_n)=K_n\times\langle z\rangle$, we have
\[
  K\cong F_n\times\Z
  \ \le\ \ H\cong F_2\times\Z
  \ \le\ \ \Gamma_2,
  \qquad
  [H:K]=n-1,
  \quad
  [\Gamma_2:H]=6.
\]
Give $K\backslash H$ the right-multiplication action and let $H$ act on $V_n$ through $\rho\circ\psi$.
The latter action is transitive by Proposition~\ref{prop:cycle}(3), and the stabiliser of $v_0$ is
$\psi\inv(\rho\inv(K_n))=K$.
Therefore the orbit map
\[
  K\backslash H\longrightarrow V_n,
  \qquad
  Kq\longmapsto v_0\mathbin{\cdot}\rho(\psi(q)),
\]
is a well-defined $H$-equivariant bijection.

\medskip
\noindent\emph{A cyclic transversal.}
We have $\rho(\psi(b^6))=g$.
Because $n$ is even, Proposition~\ref{prop:cycle} shows that $b^6$ acts as an $(n-1)$-cycle on $K\backslash H$.
Consequently
\[
  T:=\{b^{6k}\mid 0\le k\le n-2\}
\]
is a right transversal for $K$ in $H$, and $b^{6(n-1)}\in K$.

\medskip
\noindent\emph{Positivity of the Schreier elements.}
We verify the hypothesis of Lemma~\ref{lem:lifting} for the cone
$P_H=\langle Y\rangle^+$, the subgroup $K$, and the displayed
transversal $T$.
Let $t=b^{6k}$ with $0\le k\le n-2$ and let $w\in Y$.

First suppose that
$w=b^{-s}ab^r$ is one of the six elements of $Y$ containing $a$.
For some $0\le k'\le n-2$, the representative of $tw$ is $\overline{tw}=b^{6k'}$, and hence
\[
  \gamma(t,w)=tw\,\overline{tw}\inv
    =b^{6k-s}ab^{r-6k'}\in P_2
\]
by Lemma~\ref{lem:bab}.
As $\gamma(t,w)$ belongs to $K$ by definition and $K\le H$, it lies in $P_H$.

Now let $w=b^6$.
If $k<n-2$, then $tw=b^{6(k+1)}\in T$ and $\gamma(t,w)=1$.
For $k=n-2$, the representative of $tw=b^{6(n-1)}\in K$ is $1$, so
\[
  \gamma(b^{6(n-2)},b^6)=b^{6(n-1)}\in P_H.
\]
Here the last membership follows from $b^6\in Y\subseteq P_H$ and closure of $P_H$ under multiplication.
Thus every Schreier element belongs to $P_H\cup\{1\}$.

By Lemma~\ref{lem:lifting}, $K\cap P_H$ is therefore a finitely generated positive cone of
$K\cong F_n\times\Z$. Since $n\ge2$ was an arbitrary even integer,
this proves Theorem~\ref{thm:main}.
\end{proof}

\section*{Acknowledgements}
This work grew out of the author's doctoral thesis, which would not have been possible without the support and supervision of Yago Antol\'{\i}n.

\end{document}